\documentclass{article}

\usepackage[a4paper,left=2.75cm,right=2.75cm,top=3cm,bottom=3cm]{geometry} 
\usepackage[dvipsnames]{xcolor} 
\usepackage[colorlinks,linkcolor={NavyBlue},citecolor={NavyBlue},urlcolor={black}]{hyperref} 
\usepackage{graphicx,caption,subcaption} 
\usepackage{tikz}
\usetikzlibrary{math}

\usepackage{xspace} 
\usepackage{amsmath,amsthm,amsfonts,amssymb,bbm,mathrsfs,mathtools} 
\usepackage{enumitem} 

\newtheorem{theorem}{Theorem}[section]
\newtheorem{proposition}[theorem]{Proposition}
\newtheorem{lemma}[theorem]{Lemma}

\theoremstyle{remark}
\newtheorem{remark}[theorem]{Remark}

\setlist[enumerate,1]{label=(\roman*)} 

\let\originalleft\left
\let\originalright\right
\renewcommand{\left}{\mathopen{}\mathclose\bgroup\originalleft}
\renewcommand{\right}{\aftergroup\egroup\originalright}

\renewcommand{\tilde}{\widetilde}

\let\Haccent\H

\newcommand{\E}{\mathbb{E}}

\renewcommand{\H}{\mathbb{H}}

\renewcommand{\P}{\mathbb{P}}

\newcommand{\R}{\mathbb{R}}

\newcommand{\cL}{\mathcal{L}}

\newcommand{\cN}{\mathcal{N}}

\newcommand{\cZ}{\mathcal{Z}}

\newcommand{\sF}{\mathscr{F}}

\newcommand{\ie}{i.e.\@\xspace}
\newcommand{\eg}{e.g.\@\xspace}
\newcommand{\resp}{resp.\@\xspace}

\newcommand{\pp}{pp.\@\xspace}

\newcommand{\e}{\mathrm{e}}
\newcommand{\vep}{\varepsilon}
\newcommand{\1}{\mathbbm{1}}
\renewcommand{\d}[1]{\mathop{}\!\mathrm{d}#1}

\title{Almost sure path localisation for the derivative martingale of branching Brownian motion}
\author{%
  Julien~Berestycki%
\footnote{Department of Statistics, University of Oxford.}%
, Louis~Chataignier%
\footnote{Institut de Math\'ematiques de Toulouse, Universit\'e de Toulouse.}%
, and
  Gabriel~Flath\footnotemark[1]
    }
\date{\today}

\begin{document}

\maketitle

\begin{abstract}
The evolution of the front of branching Brownian motion is determined by the limit of the derivative martingale. In this work, we characterise which particles contribute to this limit. Precisely, we establish a sharp almost sure path localisation result which shows that the limit is determined by those particles whose trajectory stays within a thin tube at distance $s^{1/2}$ from the extremal particle.
\end{abstract}

\section{Introduction}

The (binary) branching Brownian motion is a continuous-time branching Markov process constructed as follows.
At time $t = 0$, a single particle starts a Brownian motion in $\R$ with variance $\sigma^2 > 0$ and drift $\rho \in \R$, from the origin.
After a random time that follows an exponential distribution with rate $\lambda > 0$, it splits into two, or equivalently, it dies and gives birth to two children.
These new particles then repeat the same process, independently of each other.
For a formal construction, see~\cite{Chauvin1991,Chataignier2024}.

We denote by $\P$ the law of the branching Brownian motion and by $\E$ the associated expectation.
Given $\gamma \ge 0$ and $x \in \R$, it is also convenient to consider $\P_{\gamma,x}$ (\resp $\P_x$) the law of the branching Brownian motion starting from position $x$ at time $\gamma \ge 0$ (\resp $\gamma = 0$), and by $\E_{\gamma,x}$ (\resp $\E_x$) the associated expectation.
With this notation, $\P = \P_{0,0}$.
We also consider $(B_t)_{t \ge 0}$ a standard Brownian motion which, under $\P_{\gamma,x}$, is at position $x$ at time $\gamma$ and is independent of the branching Brownian motion.

Let $\mathcal{N}_t$ be the set of particles at time $t$. For $v \in \mathcal{N}_t$ and $s\in \left[0,t\right]$, we say that $u\in \mathcal{N}_s$ is its ancestor at time $s$ if $v$ is a descendant of $u$ and denote it $u\preceq v$. For $u\in \mathcal{N}_t$ and $s\in \left[0,t\right]$, we write $X_u(s) \in \mathbb{R}$  for the position of particle $u$ (or its ancestor) at time $s$.
We choose to work in the setting of the boundary case, in the sense of~\cite{BigginsKyprianou2005}.
This means that we fix the parameters $\sigma^2$, $\rho$, and $\lambda$ such that
\begin{equation*}
	\E \left[ \sum_{u \in \cN_t} \e^{-X_u(t)} \right] = 1 \quad \text{and} \quad \E \left[ \sum_{u \in \cN_t} X_u(t) \e^{-X_u(t)} \right] = 0,
\end{equation*}
for any $t \ge 0$. For instance, this holds for $\sigma^2 = 1$, $\rho = 1$, and $\lambda = 1/2$ (see~\cite{AidekonBerestyckiBrunetShi2013}), which we assume from now on.

With this choice of parameters, the minimal position at time $t$ denoted by $I_t \coloneq \inf_{u \in \cN_t} X_u(t)$, grows unbounded with asymptotic speed zero:
\begin{equation}\label{eq:almost_sure_divergence} \text{almost surely }
\begin{cases}
    I_t  \xrightarrow[t \to \infty]{} \infty,\\
    I_t /t  \xrightarrow[t \to \infty]{} 0 .
\end{cases}
\end{equation}

In a seminal paper, McKean~\cite{McKean1975} showed that solutions of a certain reaction-diffusion equation -often called the F--KPP equation~\cite{Fisher1937,KolmogorovPetrovskyPiskunov1937}- can be represented in terms of expectation of the branching Brownian motion (in much the same way that solutions of the heat equation can be represented by Brownian motion expectations). This then allowed Bramson~\cite{Bramson1983} to show that $I_t$ is concentrated around $(3/2) \log t$ and that in fact
\[
I_t - \frac 32  \log t \xrightarrow[t \to \infty]{\text{d}} W,
\]
where $W$ is a random variable whose distribution corresponds to minimal travelling wave solution of the F--KPP equation, \ie $\P(W \ge -x)=w(x)$ is the unique (up to translation) solution of 
\begin{equation}\label{eq:travelling_wave}
    w'' + 2w' + w^2 - w = 0.
\end{equation}

Central to this analysis are a family of additive martingales studied by McKean~\cite{McKean1975}.
For any $\beta \ge 0,$ the process
\begin{equation}\label{eq:additive_martingale}
    W_t(\beta) = \sum_{u \in \cN_t} \e^{-\beta X_u(t) - (1-\beta)^2t/2}, \qquad t\ge 0,
\end{equation}
is a non-negative martingale and thus converges almost surely to some $W_\infty(\beta)$ as $t \to \infty$.

It is shown in McKean~\cite{McKean1975,McKean1976} and Lalley and Sellke~\cite{LalleySellke1987} that the limit $W_\infty(\beta)$ is non-trivial if and only if $\beta < 1$.
At the critical parameter $\beta =1$, Lalley and Sellke~\cite{LalleySellke1987} introduced the derivative martingale
\begin{equation}
    Z_t \coloneq -\frac{\d{W_t}}{\d{\beta}}(1) = \sum_{u \in \cN_t} X_u(t) \e^{-X_u(t)}. \label{eq:derivative_martingale}
\end{equation}
They proved the almost sure convergence of $Z_t$ to some $Z_\infty \in (0,\infty)$ and the existence of a constant $C^* > 0$ such that
\begin{equation*}
    w(x) = \P(W \ge - x) = \E \left[ \exp(-C^* Z_\infty \e^{-x}) \right].
\end{equation*}
We mention that Neveu~\cite{Neveu1988} independently proved the equality between the left-hand and right-hand sides, however, without making a link with the convergence of $I_t$.

More generally, the derivative martingale is a key tool to describe the asymptotic behavior of the branching Brownian motion.
Indeed, it appears in most of the results concerning particles close enough to $I_t$, \eg in the limit of the extremal process~\cite{ArguinBovierKistler2013,AidekonBerestyckiBrunetShi2013} or in the density of particles at sublinear distance from $I_t$~\cite{Flath2026}.
An important instance is the following Seneta--Heyde scaling,
\begin{equation}\label{eq:seneta_heyde}
    \sqrt{t} W_t(1) \xrightarrow[t \to \infty]{\P} \sqrt{\frac{2}{\pi}} Z_\infty,
\end{equation}
proved by A\"id\'ekon and Shi~\cite{AidekonShi2014} in the context of branching random walk.
A consequence of the Croft--Kingman lemma is that the convergence in probability~\eqref{eq:seneta_heyde} applies directly to branching Brownian motion (see \eg \cite[Corollary~1.A.2]{Chataignier2026}).

Moreover, in the context of branching random walk, Madaule~\cite{Madaule2016a} proved that, for any bounded continuous function $f$,
\begin{equation}\label{eq:scaling_derivative}
    \sum_{u \in \cN_t} X_u(t) \e^{-X_u(t)} f((X_u(st)/\sqrt{t})_{s \in [0,1]}) \xrightarrow[t \to \infty]{\P} Z_\infty \E[f((R_s)_{s \in [0,1]})],
\end{equation}
where $(R_s)_{s \ge 0}$ is a $3$-dimensional Bessel process.
In particular, with high probability as $t \to \infty$, particles that contribute to $Z_t$ are located at a position of order $\sqrt{t}$.
The fluctuations of the derivative martingale around its limit as well as those of~\eqref{eq:scaling_derivative} for a wide class of functions are studied by Maillard and Pain~\cite{MaillardPain2019,MaillardPain2026}.
This includes the case $f(x) = 1/x$, which provides the fluctuations of the Seneta--Heyde scaling~\eqref{eq:seneta_heyde}.

Plugging $f(x) = \1_{[a,b]}(x) $ in~\eqref{eq:scaling_derivative},  suggests that the convergence of $Z_t$ towards $Z_\infty$ should still hold (but only in probability) if one only counts particles in $[a\sqrt t, b\sqrt t]$ in the sum $Z_t$ with $a$ tending to zero and $b$ tending to infinity. We show that, in fact, one can restrict oneself to particles whose entire trajectory stays in an interval around $s \mapsto \sqrt s$ and still get an almost sure convergence. 

We introduce the following integral test for non-negative functions $a$ and $b$ from $\R_{\ge0}$ to $\R_{\ge 0}$:
\begin{equation}
\tag{IC}\label{eq:integrability_conditions}
\int^\infty \frac{a(s)}{s} \d{s} < \infty \quad \text{and} \quad \int^\infty \frac{b(s)^3}{s} \e^{-b(s)^2/2} \d{s} < \infty.
\end{equation}
The following is our main result.
\begin{theorem}\label{th:path_localisation}
Let $a,b$ be functions $\R_{\ge0} \to \R_{\ge 0}$ satisfying $a \downarrow 0$ and $b \uparrow \infty$.
Then, almost surely,
\begin{equation}\label{eq:path_localisation}
\lim_{r \to \infty} \lim_{t \to \infty} \sum_{u \in \cN_t} X_u(t) \e^{-X_u(t)} \1_{\{\forall s \in [r,t], X_u(s) \in [a(s)\sqrt{s},b(s)\sqrt{s}]\}} = \begin{cases}
    Z_\infty &\text{ if \eqref{eq:integrability_conditions} holds,} \\ 0 &\text{ otherwise.}
\end{cases} 
\end{equation}
\end{theorem}

This theorem can be viewed as the analogue of the Dvoretsky--Erd\Haccent{o}s test for $3$-dimensional Bessel processes, which states that
\begin{equation}\label{eq:dvoretsky_erdos_test}
    \lim_{r \to \infty} \P(\forall s \ge r, R_s \in [a(s)\sqrt{s},b(s)\sqrt{s}]) = 1 \text{ or } 0,
\end{equation}
depending on whether~\eqref{eq:integrability_conditions} holds or not (see \eg \cite[points 14 and 15 \pp 163--164]{ItoMcKean1974}).
This is a key result for our proof.

\begin{remark}
    \begin{enumerate}
        \item An interesting instance of Theorem~\ref{th:path_localisation} is that, for $a(s) = (\log s)^{-a}$ and $b(s) = b \sqrt{2 \log \log s}$ with $a$ and $b$ constant, the limit~\eqref{eq:path_localisation} is either $Z_\infty$ or $0$ depending on whether $a > 1$ and $b > 1$ or not.
        \item The first condition in~\eqref{eq:integrability_conditions} also appears in~\cite[Theorem~1.2]{Hu2015}.
        In it, Hu proved in the framework of branching random walk that, almost surely on the non-extinction event,
        \begin{equation*}
            \limsup_{t \to \infty} a(t) \sqrt{t} W_t(1) =
            \begin{cases}
                0 \\
                \infty
            \end{cases}
            \quad \text{if and only if} \quad
            \int^\infty \frac{a(s)}{s} \d{s}
            \begin{cases}
                < \infty, \\
                = \infty.
            \end{cases}
        \end{equation*}
        However, our proof does not use this result.
        \item The choice of a binary branching is arbitrary.
        All of the convergences of this work seem to hold for general offspring distributions in $L^1$ (possibly leading to $Z_\infty = 0$).
    \end{enumerate}
\end{remark}

One of our motivations for seeking an almost sure path localisation result for the derivative martingale stems from Conjecture~1.6 of \cite{Flath2026}.
Let $N(t,x)$ denote the number of particles below $x$ at time $t$. It is proved in \cite{Flath2026} that, whenever $x=x(t)\to\infty$ and $x=o(t)$,
\begin{equation}\label{eq:nbp}
\frac{N(t,x)}
{(2x/t)\E[N(t,x)]} \xrightarrow[t\to\infty]{\P} Z_\infty.
\end{equation}
It is further shown that this convergence cannot hold almost surely if $x \leq t^{1/3}$, while it is conjectured that almost sure convergence holds when $x$ grows faster than $t^{1/2}$.

As observed in \cite{Flath2026}, most of the particles counted by $N(t,x)$ follow, up to some time scale $\gamma^\star$, the same path localisation as the particles carrying the derivative martingale mass. This makes it natural to seek an almost sure path localisation of this mass. We hope that Theorem~\ref{th:path_localisation} provides a useful ingredient for understanding when the convergence in \eqref{eq:nbp} can be strengthened from convergence in probability to almost sure convergence, by identifying the ancestral trajectories that carry the limiting mass and clarifying the role of exceptional paths.

More generally, when a quantity related to the front of the branching Brownian motion converges in probability, it is not always clear whether it can be reinforced into an almost sure convergence or not. It is known for instance that the Seneta--Heyde scaling convergence~\eqref{eq:seneta_heyde} does not hold almost surely (\cite[Theorem~1.2]{AidekonShi2014}); 
while to the best of our knowledge, the question of whether or not the convergence~\eqref{eq:scaling_derivative} holds almost surely is open.
We hope that Theorem \ref{th:path_localisation} can be a tool to help elucidate those questions.

The rest of this paper is organised as follows: in Section~\ref{sct:stopping_lines}, we introduce the notion of stopping lines and present results reminiscent of Doob's optional stopping theorem for standard stopping times.
In Section~\ref{sct:proof_of_theorem}, we state two lemmas concerning the derivative martingale stopped at appropriate stopping lines.
These lemmas illustrate the potential of stopping lines.
The proof of our main result then follows.

Throughout this work, for $x, y \in \R$, we write $x \wedge y = \min(x,y)$ and we work with the convention $\inf \varnothing = \infty$.

\section{Stopping lines}\label{sct:stopping_lines}

To capture information on the trajectories, it is convenient to use the framework of stopping lines, a generalisation of stopping times for branching processes.
In the case of branching Brownian motion, it was introduced by Chauvin~\cite{Chauvin1991}.
Denote by $b_u$ and $d_u$ the times of birth and death of particle $u$.
A stopping line is defined as a collection of stopping times $\tau^{(u)} \in [b_u,d_u) \cup \{\infty\}$ indexed by the particles%
\footnote{%
This definition slightly differs from that of Chauvin~\cite{Chauvin1991}, who defined the stopping time of $u$ as $\tau^{(u)} - b_u$ instead.%
}%
, which satisfies the line property, that is, there is at most one particle $u$ per line of descent such that $\tau^{(u)} < \infty$.
Given such a stopping line, we consider $\cL = \{u : \tau^{(u)} < \infty\}$ the set of ``stopped particles''.
By abuse of language, we invariably speak of ``stopping line'' to refer to either the family $\tau$ of stopping times, the associated set of particles $\mathcal{L}$, or the pair $(\mathcal{L},\tau)$.
We denote by $\sF_\cL$ the associated $\sigma$-algebra, \ie the trajectorial and genealogical information about particles $u \in \cL$ from time $0$ up to $\tau^{(u)}$.

A first example of a stopping line corresponds to the population of particles alive at time $t \ge 0$, that we denote by $(\cN_t,t)$.
More precisely, it is the stopping line associated to the stopping times $t^{(u)} \coloneq t$ if $u \in \cN_t$ and $\infty$ otherwise.
The associated $\sigma$-algebra $\sF_{\cN_t} = \sF_t$ is the one containing all information available until time $t$.

We consider the same partial order as Chauvin~\cite{Chauvin1991}, defined as follows.
We set $(\cL_1,\tau_1) \preceq (\cL_2,\tau_2)$ if and only if all particles $u \in \cL_2$ are in $\cL_1$ with $\tau_1^{(u)} \le \tau_2^{(u)}$, or are strict descendants of particles in $\cL_1$.
Note that this relation is compatible with the inclusion relation for $\sigma$-algebras: if $(\cL_1,\tau_1) \preceq (\cL_2,\tau_2)$, then $\sF_{\cL_1} \subseteq \sF_{\cL_2}$.
Consequently, any non-decreasing family of stopping lines $(\cL_t,\tau_t)_{t \ge 0}$ defines a filtration $(\sF_{\cL_t})_{t \ge 0}$.

This observation is a first step to establish properties analogous to Doob's optional stopping theorem, but in the framework of stopping lines.
In short, a martingale stopped at a non-decreasing family of appropriate stopping lines is still a martingale.
To formulate precise statements, consider a process of the form
\begin{equation*}
    M_t = \sum_{u \in \cN_t} f(X_u(t)),
\end{equation*}
where $f$ is some measurable function.
Let $(\cL_t,\tau_t)_{t \ge 0}$ be a non-decreasing family of stopping lines.
We consider this process \emph{stopped} at $\tau_t$, defined by
\begin{equation*}
    M_{\tau_t} = \sum_{u \in \cL_t} f(X_u(\tau_t^{(u)})).
\end{equation*}
The following proposition is in the same spirit as~\cite[Theorem~3.1]{Chauvin1991}, where Chauvin treated the multiplicative martingales.

\begin{proposition}\label{prop:optional_stopping_theorem}
    Assume that, for all $x \in \R$, the process $(M_t)_{t \ge 0}$ is an $((\sF_t)_{t \ge 0},\P_x)$-martingale and that, for any $t \ge 0$, there exists $T \ge 0$ such that $\cL_t \preceq \cN_T$.
    Then, for any $y \in \R$, the stopped process $(M_{\tau_t})_{t \ge 0}$ is a $((\sF_{\cL_t})_{t \ge 0},\P_y)$-martingale.
\end{proposition}

\begin{proof}
    We can assume without loss of generality that $y = 0$.
    Let $s \ge 0$ and $T \ge 0$ such that $\cL_s \preceq \cN_T$.
    We can rewrite
    \begin{equation*}
        \E \left[ M_T \middle| \sF_{\cL_s} \right] = \sum_{u \in \cL_s} \E \left[ \sum_{v \in \cN_T : v \ge u} f(X_v(T)) \middle| \sF_{\cL_s} \right] = \sum_{u \in \cL_s} \phi(\tau_s^{(u)},X_u(\tau_s^{(u)})),
    \end{equation*}
    where, by the branching property of Chauvin~\cite[Proposition~2.1]{Chauvin1991},
    \begin{equation*}
        \phi(r,x) = \E_x \left[ \sum_{v \in \cN_{T-r}} f(X_v(T-r)) \right] = \E_x \left[ M_{T-r} \right] = \E_x \left[ M_0 \right] = f(x).
    \end{equation*}
    Hence, $\E[M_T|\sF_{\cL_s}] = M_{\tau_s}$.
    Now, for $t \ge s \ge 0$, we can choose $T \ge 0$ such that $\cL_s \preceq \cL_t \preceq \cN_T$ and rewrite
    \begin{equation*}
        \E \left[ M_{\tau_t} \middle| \sF_{\cL_s} \right] = \E \left[ \E \left[ M_T \middle| \sF_{\cL_t} \right] \middle| \sF_{\cL_s} \right] = \E \left[ M_T \middle| \sF_{\cL_s} \right] = M_{\tau_s},
    \end{equation*}
    which concludes.
\end{proof}

In Proposition~\ref{prop:optional_stopping_theorem}, the assumption that any $\cL_t$ is anterior to some $\cN_T$ ensures that no mass escapes from one stopping line to another.
If we only assume that no \emph{negative} mass escapes, then the martingale stopped at the stopping lines becomes a supermartingale.
This is the subject of the following proposition.

\begin{proposition}\label{prop:supermartingale}
    Consider $\cZ$ the stopping line associated with the stopping times
    \begin{equation*}
        \zeta^{(u)} \coloneq
        \begin{cases}
            \infty & \text{if $\zeta^{(v)} < \infty$ for some strict ancestor $v$ of $u$}, \\
            \inf\{t \ge 0 : f(X_u(t)) \le 0\} & \text{otherwise}.
        \end{cases}
    \end{equation*}
    Assume that, for all $x \in \R$, the process $(M_t)_{t \ge 0}$ is an $((\sF_t)_{t \ge 0},\P_x)$-martingale and that, for any $t \ge 0$, we have $\cL_t \preceq \cZ$.
    Then, for any $y \in \R$, the stopped process $(M_{\tau_t})_{t \ge 0}$ is a non-negative $((\sF_{\cL_t})_{t \ge 0},\P_y)$-supermartingale and therefore converges $\P_y$-almost surely as $t \to \infty$.
\end{proposition}

In the proof as well as in the remainder, given two stopping lines $(\mathcal{L}_1,\tau_1)$ and $(\mathcal{L}_2,\tau_2)$, it is convenient to denote by $(\cL_1 \wedge \cL_2, \tau_1 \wedge \tau_2)$ the stopping line associated to the stopping times
\begin{equation*}
    (\tau_1 \wedge \tau_2)^{(u)} \coloneq
    \begin{cases}
        \infty & \text{if $(\tau_1 \wedge \tau_2)^{(v)} < \infty$ for some strict ancestor $v$ of $u$}, \\
        \tau_1^{(u)} \wedge \tau_2^{(u)} & \text{otherwise}.
    \end{cases}
\end{equation*}

\begin{proof}
    By the Proposition~\ref{prop:optional_stopping_theorem}, for any natural numbers $n_1 < n_2$,
    \begin{equation*}
        \E_y \left[ M_{\tau_t \wedge n_2} \middle| \sF_{\cL_s \wedge n_1} \right] = M_{\tau_s \wedge n_1}.
    \end{equation*}
    The assumption that any stopping line is anterior to $\cZ$ ensures that $M_{\tau_t \wedge n_2} \ge 0$.
    By Fatou's lemma, letting $n_2 \to \infty$ and then $n_1 \to \infty$, we obtain
    \begin{equation}\label{eq:optional_stopping_theorem_fatou}
        \E_y \left[ M_{\tau_t \wedge \infty} \middle| \sF_{\cL_s \wedge \infty} \right] \le M_{\tau_s \wedge \infty},
    \end{equation}
    where $M_{\tau_t \wedge \infty} \coloneq \liminf_{n \to \infty} M_{\tau_t \wedge n}$ and $\sF_{\cL_s \wedge \infty} \coloneq \sigma(\sF_{\cL_s \wedge n} : n \ge 0)$.
    But $M_{\tau_t \wedge \infty} = M_{\tau_t} + \Delta_t$, where
    \begin{equation*}
        \Delta_t \coloneq \liminf_{n \to \infty} \sum_{u \in \cN_n : \tau_t^{(u)} > n} f(X_u(n)) \ge \Delta_s,
    \end{equation*}
    since $\zeta \ge \tau_t \ge \tau_s$.
    Moreover, note that $\Delta_s$ is $\sF_{\cL_s \wedge \infty}$-measurable.
    Therefore, the inequality~\eqref{eq:optional_stopping_theorem_fatou} becomes
    \begin{equation*}
        \E_y \left[ M_{\tau_t} \middle| \sF_{\cL_s \wedge \infty} \right] \le M_{\tau_s}.
    \end{equation*}
    Since $\sF_{\cL_s} \subseteq \sF_{\cL_s \wedge \infty}$, taking the expectation of the above inequality conditionally on $\sF_{\cL_s}$, we conclude the proof.
\end{proof}

\section{Proof of the main result}\label{sct:proof_of_theorem}

To prove our main result, we use four lemmas, that we believe to be of independent interest.
The first one tells that the derivative martingale, once stopped at appropriate stopping lines, converges almost surely to the same limit.

\begin{lemma}\label{lem:stopped_Z}
    Let $(\cL_t,\tau_t)_{t \ge 0}$ be a non-decreasing family of stopping lines such that, for any $x \in \R$, the following two conditions hold $\P_x$-almost surely:
    \begin{enumerate}
        \item\label{it:cutting_lines} any stopping line is anterior to some generation: for any $t \ge 0$, there exists $s \ge 0$ such that $(\cL_t,\tau_t) \preceq (\cN_s,s)$.
        \item\label{it:divergent_lines} any generation is anterior to some stopping line: for any $s \ge 0$, there exists $t \ge 0$ such that $(\cN_s,s) \preceq (\cL_t,\tau_t)$.
    \end{enumerate}
    Then, the stopped martingale $Z_{\tau_t}$ converges almost surely to $Z_\infty$ as $t \to \infty$.
\end{lemma}

Although we believe that the tools developed in Section~\ref{sct:stopping_lines} are useful to prove Lemma~\ref{lem:stopped_Z}, here we present a short proof that does not use them.
Instead, we prove that it is a consequence of Chauvin~\cite[Theorem~3.4]{Chauvin1991}.
It would be interesting to extend Lemma~\ref{lem:stopped_Z} to martingales more general than the derivative martingale.

\begin{proof}
    We have seen in~\eqref{eq:travelling_wave} that the following function is a F--KPP travelling wave at critical speed (unique up to translation),
    \begin{equation*}
        w(x) = \lim_{t \to \infty} \P(I_t \ge (3/2) \log t - x).
    \end{equation*}
    By Lalley and Sellke~\cite[Equation~(25)]{LalleySellke1987},
    \begin{equation*}
        \Pi_t \coloneq \prod_{u \in \cN_t} w(X_u(t)) \xrightarrow[t \to \infty]{} \e^{-C^* Z_\infty}, \quad \text{almost surely},
    \end{equation*}
    where $C^* > 0$.
    The process $\Pi_t$ is known in the literature as a \emph{multiplicative martingale}.

    Our conditions~\ref{it:cutting_lines} and~\ref{it:divergent_lines} are equivalent to the conditions~(3.3) and~(3.9) of Chauvin~\cite{Chauvin1991}.
    By Theorem~3.4 of the latter, they ensure that we also have
    \begin{equation}\label{eq:stopped_multiplicative}
        \Pi_{\tau_t} \coloneq \prod_{u \in \cL_t} w(X_u(\tau_t^{(u)})) \xrightarrow[t \to \infty]{} \e^{-C^* Z_\infty}, \quad \text{almost surely}.
    \end{equation}
    But, by~\cite[Equation~(22)]{LalleySellke1987},
    \begin{equation*}
        1 - w(x) = (1 + \vep(x)) C^* x \e^{-x},
    \end{equation*}
    where $\vep(x)$ is some function that converges to $0$ as $x \to \infty$.
    Therefore,
    \begin{equation*}
        \Pi_{\tau_t} = \exp \sum_{u \in \cL_t} \log(1 - (1 + \vep(X_u(\tau_t^{(u)}))) C^* X_u(\tau_t^{(u)}) \e^{-X_u(\tau_t^{(u)})}).
    \end{equation*}
    Since $\cL_t$ is posterior to any given generation for $t$ sufficiently large and since the minimal position of the branching Brownian motion diverges almost surely to $\infty$, it follows that
    \begin{equation}\label{eq:estimate_multiplicative}
        \Pi_{\tau_t} = \e^{-(1+\vep_t)C^* Z_{\tau_t}}, \quad \text{almost surely},
    \end{equation}
    where $\vep_t$ is some process that converges almost surely to $0$ as $t \to \infty$.
    The conclusion follows from~\eqref{eq:stopped_multiplicative} and~\eqref{eq:estimate_multiplicative}.  
\end{proof}

Now, to fit into the framework of Theorem~\ref{th:path_localisation}, let $a(t)$ and $b(t)$ be two non-negative functions such that $a(t) \downarrow 0$ and $b(t) \uparrow \infty$ as $t \uparrow \infty$.
Consider the stopping line $(\cL_r,\tau_r)$ associated to the stopping times
\begin{equation}\label{eq:def_stopping_line}
    \tau_r^{(u)} \coloneq
    \begin{cases}
        \infty & \text{if $\tau_r^{(v)} < \infty$ for some strict ancestor $v$ of $u$}, \\
        \inf\{s \ge r : X_u(s) \notin (a(s)\sqrt{s},b(s)\sqrt{s})\} & \text{otherwise}.
    \end{cases}
\end{equation}

Recall that, with the notation introduced in Section~\ref{sct:stopping_lines}, the derivative martingale stopped at $(\cL_r,\tau_r)$ (see Figure~\ref{fig:bbm_stopped}) is
\begin{equation}\label{eq:def_stopped_derivative_martingale}
    Z_{\tau_r} = \sum_{u \in \cL_r} X_u(\tau_r^{(u)}) \e^{-X_u(\tau_r^{(u)})}.
\end{equation}
We first verify that this stopped martingale admits an almost sure limit.

\begin{lemma}\label{lem:supermartingale}
    The stopped martingale $Z_{\tau_r}$ converges almost surely to some random variable $Z_{\tau_\infty}$ as $r \to \infty$.
\end{lemma}

\begin{proof}
    Fix an integer $n \ge 0$ and define, for $t \ge n$,
    \begin{equation*}
        \tilde{Z}_t^{(n)} = \sum_{u \in \cN_t} X_u(t) \e^{-X_u(t)} \1_{\{X_u(n) > 0\}}.
    \end{equation*}
    The process $(\tilde{Z}_t^{(n)})_{t \ge n}$ is an $(\sF_t)_{t \ge n}$-martingale.
    In order to reduce it to a non-negative process, consider the stopping line $(\cZ,\zeta)$ associated to the stopping times
    \begin{equation*}
        \zeta^{(u)} \coloneq
        \begin{cases}
            \infty & \text{if $\zeta^{(v)} < \infty$ for some strict ancestor $v$ of $u$}, \\
            \inf\{s \ge n : X_u(s) \le 0\} & \text{otherwise}.
        \end{cases}
    \end{equation*}
    By Proposition~\ref{prop:supermartingale}, the process $(\tilde{Z}_{\tau_r \wedge \zeta}^{(n)})_{r \ge n}$ is a non-negative $(\sF_{\cL_r \wedge \cZ})_{r \ge n}$-supermartingale and therefore converges almost surely as $r \to \infty$.
    It follows that, with probability $1$, for any integer $n \ge 0$, the process $(\tilde{Z}_{\tau_r \wedge \zeta}^{(n)})_{r \ge n}$ converges.
    Denote by $\Omega'$ this almost sure event.

    Further consider the event $\Omega''$ where there exists an integer $n_0 \ge 0$ such that, for any $t \ge n_0$, all the positions $X_u(t)$ are non-negative.
    By~\eqref{eq:almost_sure_divergence}, $\Omega''$ has probability $1$ too.
    Therefore, it is also the case of $\Omega' \cap \Omega''$.
    But, on this event, $Z_{\tau_r} = \tilde{Z}_{\tau_r \wedge \zeta}^{(n_0)}$ for $r \ge n_0$ and therefore converges.
    In conclusion, $Z_{\tau_r}$ converges almost surely as $r \to \infty$.
\end{proof}

Now consider the stopping line $(\cL_r \wedge t, \tau_r \wedge t)$ and the associated stopped martingale (see Figure~\ref{fig:bbm_stopped})
\begin{equation*}
    Z_{\tau_r \wedge t} = \sum_{u \in \cL_r} X_u(\tau_r^{(u)}) \e^{-X_u(\tau_r^{(u)})} \1_{\{\tau_r^{(u)} < t\}} + \sum_{u \in \cN_t} X_u(t) \e^{-X_u(t)} \1_{\{\tau_r^{(u)} \ge t\}}.
\end{equation*}

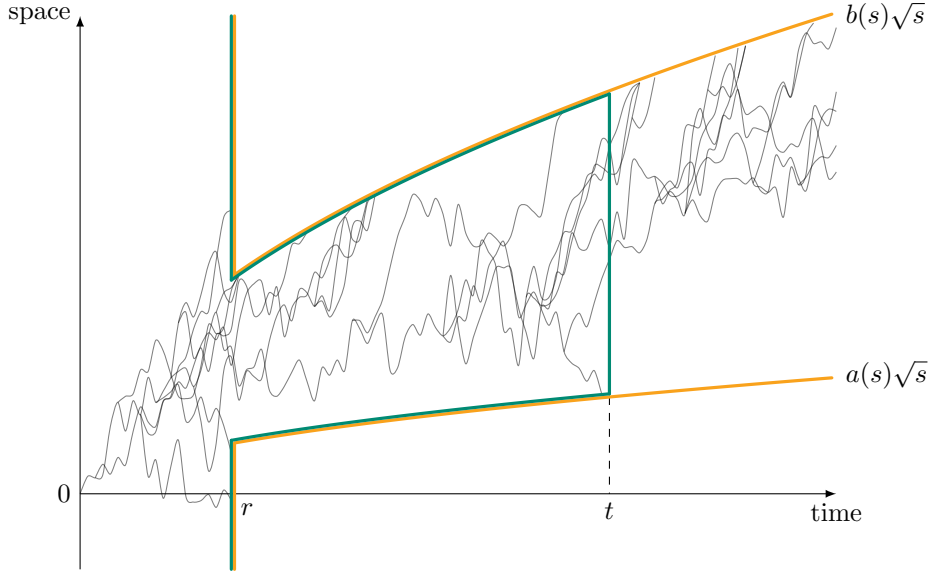
\begin{figure}[!htbp]
    \centering
    \begin{tikzpicture}[xscale=1,yscale=1]
    \tikzmath{\horizon=10;\t=7;\r=2;\a=.5;\b=2;\ymin=-1;\ymax=sqrt(\horizon)*\b;}
    
    \draw[->,>=latex] (0,0) node[left]{$0$} -- (\horizon,0) node[below]{time};
    \draw[->,>=latex] (0,\ymin) -- (0,\ymax) node[left]{space};
    
    \draw[thin,opacity=.5] plot[smooth] coordinates {(0.0,0.00) (0.1,0.23) (0.2,0.21) (0.3,0.25) (0.4,0.89) (0.5,1.21) (0.6,0.83) (0.7,1.23) (0.8,1.05) (0.9,1.54) (1.0,1.21) (1.1,1.36) (1.2,1.53) (1.3,1.62) (1.4,1.66) (1.5,1.42) (1.6,1.64) (1.7,2.10) (1.8,2.40) (1.9,1.91) (2.0,1.94) (2.1,1.43) (2.2,1.64) (2.3,1.68) (2.4,1.91) (2.5,1.38) (2.6,1.29) (2.7,1.42) (2.8,1.61) (2.9,1.77) (3.0,1.49) (3.1,1.23) (3.2,1.67) (3.3,1.57) (3.4,1.77) (3.5,1.63) (3.6,2.29) (3.7,2.18) (3.8,1.99) (3.9,2.23) (4.0,2.37) (4.1,1.63) (4.2,1.67) (4.3,2.08) (4.4,2.24) (4.5,2.10) (4.6,2.38) (4.7,2.20) (4.8,2.09) (4.9,2.20) (5.0,2.44) (5.1,2.83) (5.2,2.63) (5.3,3.34) (5.4,3.20) (5.5,3.00) (5.6,2.17) (5.7,2.18) (5.8,1.70) (5.9,2.60) (6.0,2.59) (6.1,2.49) (6.2,2.51) (6.3,2.65) (6.4,2.96) (6.5,3.45) (6.6,3.25) (6.7,3.24) (6.8,3.49) (6.9,3.63) (7.0,3.28) (7.1,3.73) (7.2,3.38) (7.3,4.38) (7.4,4.33) (7.5,4.18) (7.6,4.20) (7.7,4.11) (7.8,3.95) (7.9,3.94) (8.0,4.02) (8.1,4.18) (8.2,4.67) (8.3,4.89) (8.4,5.05) (8.5,5.50) (8.6,5.69) (8.7,5.34) (8.8,5.03) (8.9,5.07) (9.0,5.13) (9.1,4.99) (9.2,4.60) (9.3,4.52) (9.4,4.63) (9.5,4.60) (9.6,4.86) (9.7,4.96) (9.8,4.59) (9.9,4.17) (10.0,4.26)};
    \draw[thin,opacity=.5] plot[smooth] coordinates {(0.2,0.21) (0.3,0.39) (0.4,0.26) (0.5,0.48) (0.6,0.38) (0.7,0.82) (0.8,1.15) (0.9,0.95) (1.0,0.89) (1.1,1.23) (1.2,1.27) (1.3,1.48) (1.4,1.40) (1.5,1.69) (1.6,1.97) (1.7,2.13) (1.8,2.04) (1.9,2.20) (2.0,1.70) (2.1,2.26) (2.2,2.65) (2.3,2.78) (2.4,2.81) (2.5,2.37) (2.6,2.74) (2.7,2.08) (2.8,2.46) (2.9,2.60) (3.0,2.63) (3.1,2.93) (3.2,2.86) (3.3,3.10) (3.4,3.51) (3.5,3.67) (3.6,3.79)};
    \draw[thin,opacity=.5] plot[smooth] coordinates {(0.5,1.21) (0.6,1.29) (0.7,1.32) (0.8,1.64) (0.9,1.62) (1.0,2.03) (1.1,1.89) (1.2,1.58) (1.3,1.29) (1.4,1.26) (1.5,1.31) (1.6,1.70) (1.7,2.24) (1.8,2.12) (1.9,2.20) (2.0,2.69) (2.1,2.90)};
    \draw[thin,opacity=.5] plot[smooth] coordinates {(0.8,1.05) (0.9,0.23) (1.0,0.45) (1.1,-0.06) (1.2,0.36) (1.3,0.71) (1.4,-0.04) (1.5,-0.11) (1.6,-0.09) (1.7,-0.01) (1.8,-0.18) (1.9,0.05) (2.0,-0.09)};
    \draw[thin,opacity=.5] plot[smooth] coordinates {(0.8,1.05) (0.9,0.80) (1.0,0.82) (1.1,1.22) (1.2,1.29) (1.3,1.67) (1.4,1.81) (1.5,2.18) (1.6,2.41) (1.7,2.42) (1.8,3.09) (1.9,3.66) (2.0,3.75)};
    \draw[thin,opacity=.5] plot[smooth] coordinates {(1.0,1.21) (1.1,1.27) (1.2,1.35) (1.3,1.44) (1.4,0.96) (1.5,1.04) (1.6,1.28) (1.7,0.95) (1.8,1.13) (1.9,0.90) (2.0,0.54)};
    \draw[thin,opacity=.5] plot[smooth] coordinates {(1.2,1.53) (1.3,2.26) (1.4,2.53) (1.5,2.34) (1.6,2.64) (1.7,3.04) (1.8,3.20) (1.9,2.68) (2.0,2.54) (2.1,2.90)};
    \draw[thin,opacity=.5] plot[smooth] coordinates {(1.3,2.26) (1.4,2.40) (1.5,1.80) (1.6,1.93) (1.7,1.62) (1.8,2.10) (1.9,2.09) (2.0,2.18) (2.1,2.81) (2.2,2.75) (2.3,2.36) (2.4,2.41) (2.5,2.66) (2.6,2.62) (2.7,2.86) (2.8,3.29) (2.9,3.31) (3.0,3.46)};
    \draw[thin,opacity=.5] plot[smooth] coordinates {(1.5,1.42) (1.6,0.95) (1.7,0.30) (1.8,0.67) (1.9,1.63) (2.0,1.81) (2.1,1.24) (2.2,1.62) (2.3,1.55) (2.4,1.52) (2.5,2.01) (2.6,2.35) (2.7,2.50) (2.8,2.29) (2.9,2.70) (3.0,2.66) (3.1,2.48) (3.2,2.62) (3.3,2.43) (3.4,3.10) (3.5,3.66) (3.6,3.37) (3.7,3.37) (3.8,3.85) (3.9,3.94) (4.0,4.00)};
    \draw[thin,opacity=.5] plot[smooth] coordinates {(3.1,2.48) (3.2,2.55) (3.3,2.34) (3.4,2.50) (3.5,2.79) (3.6,3.53) (3.7,3.61) (3.8,3.49) (3.9,3.95)};
    \draw[thin,opacity=.5] plot[smooth] coordinates {(3.1,2.48) (3.2,3.04) (3.3,3.24) (3.4,3.29) (3.5,3.14) (3.6,3.17) (3.7,3.50) (3.8,3.90)};
    \draw[thin,opacity=.5] plot[smooth] coordinates {(1.8,2.40) (1.9,2.66) (2.0,2.73) (2.1,2.90)};
    \draw[thin,opacity=.5] plot[smooth] coordinates {(3.6,2.29) (3.7,2.29) (3.8,2.13) (3.9,2.34) (4.0,2.57) (4.1,3.00) (4.2,3.41) (4.3,3.67) (4.4,3.96) (4.5,3.88) (4.6,3.67) (4.7,3.91) (4.8,3.70) (4.9,3.28) (5.0,3.83) (5.1,3.45) (5.2,3.28) (5.3,3.29) (5.4,3.45) (5.5,3.38) (5.6,3.28) (5.7,3.42) (5.8,3.79) (5.9,3.82) (6.0,3.73) (6.1,3.69) (6.2,4.28) (6.3,4.90) (6.4,5.06)};
    \draw[thin,opacity=.5] plot[smooth] coordinates {(4.8,2.09) (4.9,2.01) (5.0,2.34) (5.1,2.30) (5.2,2.75) (5.3,2.54) (5.4,3.00) (5.5,3.17) (5.6,3.34) (5.7,2.92) (5.8,2.45) (5.9,2.13) (6.0,1.79) (6.1,2.47) (6.2,2.46) (6.3,2.92) (6.4,2.91) (6.5,3.33) (6.6,3.69) (6.7,3.92) (6.8,4.34) (6.9,4.32) (7.0,4.24) (7.1,4.25) (7.2,4.52) (7.3,4.98) (7.4,5.44)};
    \draw[thin,opacity=.5] plot[smooth] coordinates {(4.8,2.09) (4.9,2.89) (5.0,2.21) (5.1,1.72) (5.2,2.07) (5.3,2.36) (5.4,2.65) (5.5,2.99) (5.6,2.99) (5.7,2.97) (5.8,2.34) (5.9,2.51) (6.0,2.67) (6.1,2.44) (6.2,2.23) (6.3,2.40) (6.4,1.88) (6.5,1.73) (6.6,1.66) (6.7,1.49) (6.8,1.63) (6.9,1.31)};
    \draw[thin,opacity=.5] plot[smooth] coordinates {(6.4,1.88) (6.5,2.24) (6.6,2.39) (6.7,2.21) (6.8,2.53) (6.9,2.85) (7.0,3.10) (7.1,3.35) (7.2,3.26) (7.3,3.27) (7.4,3.19) (7.5,3.38) (7.6,3.84) (7.7,3.90) (7.8,4.02) (7.9,4.14) (8.0,4.83) (8.1,4.87) (8.2,4.87) (8.3,5.35) (8.4,5.80)};
    \draw[thin,opacity=.5] plot[smooth] coordinates {(7.6,3.84) (7.7,4.03) (7.8,4.14) (7.9,4.05) (8.0,4.32) (8.1,4.37) (8.2,3.95) (8.3,4.53) (8.4,4.57) (8.5,4.99) (8.6,5.26) (8.7,5.90)};
    \draw[thin,opacity=.5] plot[smooth] coordinates {(8.3,5.35) (8.4,5.27) (8.5,5.39) (8.6,5.25) (8.7,5.50) (8.8,5.93)};
    \draw[thin,opacity=.5] plot[smooth] coordinates {(8.7,5.50) (8.8,5.93)};
    \draw[thin,opacity=.5] plot[smooth] coordinates {(7.3,4.98) (7.4,4.84) (7.5,5.04) (7.6,5.51)};
    \draw[thin,opacity=.5] plot[smooth] coordinates {(5.9,2.60) (6.0,3.01) (6.1,2.60) (6.2,2.60) (6.3,2.88) (6.4,3.57) (6.5,3.87) (6.6,3.97) (6.7,4.10) (6.8,4.23) (6.9,4.16) (7.0,4.59) (7.1,4.83) (7.2,5.17) (7.3,5.37) (7.4,5.44)};
    \draw[thin,opacity=.5] plot[smooth] coordinates {(6.1,2.49) (6.2,3.00) (6.3,3.21) (6.4,3.46) (6.5,3.48) (6.6,3.79) (6.7,4.19) (6.8,4.29) (6.9,4.69) (7.0,4.60) (7.1,4.09) (7.2,4.70) (7.3,5.25) (7.4,5.44)};
    \draw[thin,opacity=.5] plot[smooth] coordinates {(7.6,4.20) (7.7,3.84) (7.8,4.00) (7.9,3.89) (8.0,3.71) (8.1,4.76) (8.2,5.10) (8.3,4.77) (8.4,4.63) (8.5,5.00) (8.6,4.97) (8.7,5.35) (8.8,4.73) (8.9,4.84) (9.0,4.82) (9.1,4.80) (9.2,5.06) (9.3,5.40) (9.4,5.25) (9.5,5.65) (9.6,6.07) (9.7,6.23)};
    \draw[thin,opacity=.5] plot[smooth] coordinates {(8.0,3.71) (8.1,3.89) (8.2,3.68) (8.3,3.61) (8.4,3.99) (8.5,4.15) (8.6,4.17) (8.7,4.17) (8.8,4.10) (8.9,4.18) (9.0,4.01) (9.1,4.00) (9.2,4.48) (9.3,4.33) (9.4,4.45) (9.5,4.18) (9.6,4.01) (9.7,4.47) (9.8,4.95) (9.9,5.09) (10.0,5.06)};
    \draw[thin,opacity=.5] plot[smooth] coordinates {(8.0,3.71) (8.1,4.04) (8.2,4.33) (8.3,4.45) (8.4,4.38) (8.5,4.43) (8.6,4.69) (8.7,4.77) (8.8,4.56) (8.9,4.11) (9.0,4.06) (9.1,3.99) (9.2,4.05) (9.3,3.66) (9.4,4.32) (9.5,4.09) (9.6,4.23) (9.7,5.11) (9.8,4.98) (9.9,4.88) (10.0,5.32)};
    \draw[thin,opacity=.5] plot[smooth] coordinates {(8.9,4.11) (9.0,4.40) (9.1,4.80) (9.2,5.26) (9.3,5.45) (9.4,5.26) (9.5,5.87) (9.6,6.10) (9.7,5.96) (9.8,5.98) (9.9,5.89) (10.0,6.17)};
    \draw[thin,opacity=.5] plot[smooth] coordinates {(9.5,4.09) (9.6,4.19) (9.7,4.37) (9.8,4.46) (9.9,4.45) (10.0,4.61)};
    
    \draw[dashed] (\t,{sqrt(\t)*\a-.042}) -- (\t,0) node[below]{$t$};
    \draw (\horizon,{sqrt(\horizon)*\a}) node[right]{$a(s)\sqrt{s}$};
    \draw (\horizon,{sqrt(\horizon)*\b}) node[right]{$b(s)\sqrt{s}$};
    
    \draw[very thick,YellowOrange,line cap=round] (\r+.042,{sqrt(\r)*\a-.042}) -- (\r+.042,\ymin);
    \draw (\r,0) node[below right]{$r$};
    \draw[very thick,YellowOrange,domain=\r+.042:\horizon,samples={(\horizon-(\r+.042))*10},line cap=round]
        plot (\x,{sqrt(\x)*\a-.042});
    \draw[very thick,YellowOrange,line cap=round] (\r+.042,{sqrt(\r)*\b+.042}) -- (\r+.042,\ymax);
    \draw[very thick,YellowOrange,domain=\r+.042:\horizon,samples={(\horizon-(\r+.042))*10},line cap=round]
        plot (\x,{sqrt(\x)*\b+.042});
    
    \draw[very thick,PineGreen,line cap=round] (\r,{sqrt(\r)*\a}) -- (\r,\ymin);
    \draw[very thick,PineGreen,domain=\r:\t,samples={(\t-\r)*10},line cap=round]
        plot (\x,{sqrt(\x)*\a});
    \draw[very thick,PineGreen,line cap=round] (\r,{sqrt(\r)*\b}) -- (\r,\ymax);
    \draw[very thick,PineGreen,domain=\r:\t,samples={(\t-\r)*10},line cap=round]
        plot (\x,{sqrt(\x)*\b});
    \draw[very thick,PineGreen,line cap=round] (\t,{sqrt(\t)*\a}) -- (\t,{sqrt(\t)*\b});
    
    \end{tikzpicture}
    \caption{Representation of a branching Brownian motion with particles stopped at two different stopping lines: the yellow curve corresponds to $\cL_r$, the green one to $\cL_r \wedge t$.}
    \label{fig:bbm_stopped}
\end{figure}

The following lemma is useful to determine the almost sure limit $Z_{\tau_\infty}$ and to analyse the left-hand side of~\eqref{eq:path_localisation}.

\begin{lemma}\label{lem:two_step_limit}
    We have $\lim_{r \to \infty} \lim_{t \to \infty} Z_{\tau_r \wedge t} = Z_\infty$ almost surely.
\end{lemma}

\begin{proof}
    By Lemma~\ref{lem:stopped_Z}, the above almost sure limit holds if we make $r$ depend on $t$ and increase to $\infty$ as $t \uparrow \infty$.
    Indeed, for such a choice of $r = r_t$, the family of stopping lines $(\cL_r \wedge t)_{t \ge 0}$ is non-decreasing, the relation $\cL_r \wedge t \preceq \cN_t$ implies~\ref{it:cutting_lines}, and the relation $\cN_r \preceq \cL_r \wedge t$ implies~\ref{it:divergent_lines}.
    In particular, the stopped martingale $Z_{\tau_r \wedge t}$ converges in probability to $Z_\infty$ as $t \to \infty$.
    Since this holds for any choice of $r = r_t$ that increases to $\infty$, a diagonal argument (see \eg \cite[Lemma~A.1]{ChataignierLuo2026}) shows that, for any $\vep > 0$,
    \begin{equation}\label{eq:two_step_limit_proba}
        \limsup_{r \to \infty} \limsup_{t \to \infty} \P(|Z_{\tau_r \wedge t} - Z_\infty| > \vep) = 0.
    \end{equation}
    It remains to strengthen this into an almost sure convergence.
    
    To this end, we show that the following limit exists almost surely:
    \begin{equation*}
        \tilde{Z} \coloneq \lim_{r \to \infty} \lim_{t \to \infty} Z_{\tau_r \wedge t}.
    \end{equation*}
    Fix $r \ge 0$.
    By Proposition~\ref{prop:optional_stopping_theorem}, the process $(Z_{\tau_r \wedge t})_{t \ge r}$ is an $(\sF_{\cL_r \wedge t})_{t \ge r}$-martingale.
    Furthermore, conditionally on $\sF_{\cL_r \wedge r} = \sF_r$, it is bounded from below:
    \begin{equation*}
        Z_{\tau_r \wedge t} \ge \sum_{u \in \cN_r : X_u(r) < 0} X_u(r) \e^{-X_u(r)},
    \end{equation*}
    Therefore, it converges almost surely as $t \to \infty$ to some limit $Z_{\tau_r \wedge \infty}$.
    To show that the latter converges almost surely as $r \to \infty$, we decompose it as
    \begin{equation*}
        Z_{\tau_r \wedge \infty} = Z_{\tau_r} + (Z_{\tau_r \wedge \infty} - Z_{\tau_r}).
    \end{equation*}
    By Lemma~\ref{lem:supermartingale} the term $Z_{\tau_r}$ converges almost surely as $r \to \infty$.
    Regarding the second term, it has the following almost sure expression:
    \begin{equation*}
        Z_{\tau_r \wedge \infty} - Z_{\tau_r} = \lim_{t \to \infty} \sum_{u \in \cN_t} X_u(t) \e^{-X_u(t)} \1_{\{\forall s \in [r,t], X_u(s) \in [a(s)\sqrt{s},b(s)\sqrt{s}]\}}.
    \end{equation*}
    By~\eqref{eq:almost_sure_divergence}, almost surely, for $t$ large enough, all the terms of the above sum are non-negative.
    In particular, almost surely, $Z_{\tau_r \wedge \infty} - Z_{\tau_r}$ is non-decreasing in $r$ as limit of non-decreasing functions, and then converges in $\R \cup \{\infty\}$ as $r \to \infty$.
    Hence, $\tilde{Z}$ exists almost surely in $\R \cup \{\infty\}$.
    
    Now, for any $\vep > 0$,
    \begin{equation}\label{eq:two_step_limit_as}
        \P(|\tilde{Z} - Z_\infty| > \vep) = \P(\lim_{r \to \infty} \lim_{t \to \infty} |Z_{\tau_r \wedge t} - Z_\infty| > \vep) \le \liminf_{r \to \infty} \liminf_{t \to \infty} \P(|Z_{\tau_r \wedge t} - Z_\infty| > \vep),
    \end{equation}
    by Fatou's lemma.
    
    The conclusion follows from~\eqref{eq:two_step_limit_proba} and~\eqref{eq:two_step_limit_as}.
\end{proof}

Recall that $Z_{\tau_\infty}$ is the almost sure limit of $Z_{\tau_r}$, defined in~\eqref{eq:def_stopped_derivative_martingale}, as $r \to \infty$.
The following lemma identifies the mass captured by a stopping line, depending on whether the integrability conditions in~\eqref{eq:integrability_conditions} hold or not.

\begin{lemma}\label{lem:0_1_law}
Almost surely, the limiting mass satisfies
\begin{equation*}
    Z_{\tau_\infty} = 
    \begin{cases} 
        0 & \text{if the integrability conditions in~\eqref{eq:integrability_conditions} hold,} \\ 
        Z_\infty & \text{otherwise.} 
    \end{cases}
\end{equation*}
\end{lemma}

\begin{proof}
    Note that, by~\eqref{eq:almost_sure_divergence}, there exists almost surely some time from which all particle positions are positive.
    In particular, almost surely, for $r$ and $t$ large enough, we have
    \begin{equation*}
        0 \le \sum_{u \in \cL_r} X_u(\tau_r^{(u)}) \e^{-X_u(\tau_r^{(u)})} \1_{\{\tau_r^{(u)} < t\}} \le Z_{\tau_r \wedge t}.
    \end{equation*}
    Letting $t \to \infty$ and then $r \to \infty$, we deduce from Lemmas~\ref{lem:supermartingale} and~\ref{lem:two_step_limit} that $0\leq Z_{\tau_\infty}\leq Z_\infty$ almost surely.
    Therefore, to prove that $Z_{\tau_\infty} = 0$ (\resp $Z_{\tau_\infty} = Z_\infty$), it suffices to bound $\E[Z_{\tau_\infty}] \le 0$ (\resp $\E[Z_\infty - Z_{\tau_\infty}] \le 0$).

    First, suppose that the integrability conditions hold.
    \begin{align*}
    \E\left[Z_{\tau_\infty}\right]
    & = \E\left[\liminf_{r \to \infty} Z_{\tau_r}\right] \\
    & \le \E\left[\liminf_{\ell \to \infty} \liminf_{r \to \infty} \sum_{u \in \cL_r} (X_u(\tau_r^{(u)}) + \ell) \e^{-X_u(\tau_r^{(u)})} \1_{\{\underline{X}_u(\tau_r^{(u)}) \ge -\ell\}}\right] \\
    & \le \liminf_{\ell \to \infty} \liminf_{r \to \infty} \E\left[\sum_{u \in \cL_r} (X_u(\tau_r^{(u)}) + \ell) \e^{-X_u(\tau_r^{(u)})} \1_{\{\underline{X}_u(\tau_r^{(u)}) \ge -\ell\}}\right],
    \end{align*}
    by Fatou's lemma.
    By the analogue of the many-to-one formula for stopping lines,
    \begin{align*}
        \E\left[\sum_{u \in \cL_r} (X_u(\tau_r^{(u)}) + \ell) \e^{-X_u(\tau_r^{(u)})} \1_{\{\underline{X}_u(\tau_r^{(u)}) \ge -\ell\}}\right] &= \E\left[(B_{\tau_r} + \ell) \1_{\{\underline{B}_{\tau_r} \ge -\ell\}} \1_{\{\tau_r < \infty\}}\right]\\
        &=\ell \P_\ell\left(\exists s \ge r, R_s-\ell \notin [a(s)\sqrt{s},b(s)\sqrt{s}]\right),
    \end{align*}
    where $(R_s)_{s\geq0}$ is a $3$-dimensional Bessel process. Moreover,
    \begin{equation*}
        \lim_{r \to \infty} \P_\ell\left(\exists s \ge r, R_s-\ell \notin [a(s)\sqrt{s},b(s)\sqrt{s}]\right)= \P_\ell\left(R_s -\ell \notin [a(s)\sqrt{s},b(s)\sqrt{s}] \text{ infinitely often}\right). 
    \end{equation*}
    Therefore, if the integrability conditions in~\eqref{eq:integrability_conditions} hold,
    \begin{equation*}
        \E\left[Z_{\tau_\infty}\right]\leq \liminf_{\ell \to \infty}\ell \P_\ell\left(R_s -\ell \notin [a(s)\sqrt{s},b(s)\sqrt{s}] \text{ infinitely often}\right)=0,
    \end{equation*}
    by the Dvoretsky--Erd\Haccent{o}s test~\eqref{eq:dvoretsky_erdos_test}, and $Z_{\tau_\infty}=0$ almost surely.
    
    On the other hand, by Lemmas~\ref{lem:supermartingale} and~\ref{lem:two_step_limit},
    \begin{align*}
    \E[Z_\infty-Z_{\tau_\infty}] &= \E\left[\liminf_{r \to \infty} \liminf_{t \to \infty} (Z_{\tau_r \wedge t} - \sum_{u \in \cL_r} X_u(\tau_r^{(u)}) \e^{-X_u(\tau_r^{(u)})} \1_{\{\tau_r^{(u)} < t\}})\right] \\
    &=
    \E\left[
    \liminf_{r\to\infty}\liminf_{t\to\infty}
    \sum_{u\in\cN_t}
    X_u(t)\e^{-X_u(t)}
    \1_{\{
    \forall s\in[r,t],\,
    X_u(s)\in[a(s)\sqrt{s},b(s)\sqrt{s}]
    \}}\right]
    \\
    & \le \E\left[\liminf_{\ell \to \infty} \liminf_{r \to \infty} \liminf_{t \to \infty} \sum_{u \in \cN_t} (X_u(t) + \ell) \e^{-X_u(t)} \1_{\{\underline{X}_u(t) \ge -\ell\}} \1_{\{\forall s \in [r,t], X_u(s) \in [a(s)\sqrt{s},b(s)\sqrt{s}]\}}\right] \\
    & \le \liminf_{\ell \to \infty} \liminf_{r \to \infty} \liminf_{t \to \infty} \E\left[\sum_{u \in \cN_t} (X_u(t) + \ell) \e^{-X_u(t)} \1_{\{\underline{X}_u(t) \ge -\ell\}} \1_{\{\forall s \in [r,t], X_u(s) \in [a(s)\sqrt{s},b(s)\sqrt{s}]\}}\right].
    \end{align*} 
    Applying the many-to-one formula once again,
    \begin{align*}
        &\E\left[\sum_{u \in \cN_t} (X_u(t) + \ell) \e^{-X_u(t)} \1_{\{\underline{X}_u(t) \ge -\ell\}} \1_{\{\forall s \in [r,t], X_u(s) \in [a(s)\sqrt{s},b(s)\sqrt{s}]\}}\right] \\
        & \qquad = \E\left[(B_{t} + \ell) \1_{\{\underline{B}_{t} \ge -\ell\}}\1_{\{\forall s \in [r,t], B_s \in [a(s)\sqrt{s},b(s)\sqrt{s}]\}}\right] \\
        & \qquad = \ell \P_\ell\left(\forall s \in [r,t], R_s-\ell \in [a(s)\sqrt{s},b(s)\sqrt{s}]\right).
    \end{align*}
     As before,
    \begin{equation*}
        \lim_{r \to \infty}\lim_{t \to \infty} \P_\ell\left(\forall s \in [r,t], R_s-\ell \in [a(s)\sqrt{s},b(s)\sqrt{s}]\right)= \P_\ell\left(R_s -\ell \in [a(s)\sqrt{s},b(s)\sqrt{s}] \text{ eventually}\right). 
    \end{equation*}
    Therefore, if the integrability conditions in~\eqref{eq:integrability_conditions} do not hold,
    \begin{equation*}
        \E[Z_\infty-Z_{\tau_\infty}] \le \liminf_{\ell \to \infty}\ell \P_\ell\left(R_s -\ell \in [a(s)\sqrt{s},b(s)\sqrt{s}] \text{ eventually}\right)=0,
    \end{equation*}
    by the Dvoretsky--Erd\Haccent{o}s test~\eqref{eq:dvoretsky_erdos_test}, and $Z_{\tau_\infty}=Z_\infty$ almost surely.
\end{proof}

We are now able to prove the main result.

\begin{proof}[Proof of Theorem~\ref{th:path_localisation}]
    The left-hand side of~\eqref{eq:path_localisation} can be expressed as
    \begin{multline*}
        \lim_{r \to \infty} \lim_{t \to \infty} \sum_{u \in \cN_t} X_u(t) \e^{-X_u(t)} \1_{\{\forall s \in [r,t], X_u(s) \in [a(s)\sqrt{s},b(s)\sqrt{s}]\}} \\
        = \lim_{r \to \infty} \lim_{t \to \infty} (Z_{\tau_r \wedge t} - \sum_{u \in \cL_r} X_u(\tau_r^{(u)}) \e^{-X_u(\tau_r^{(u)})} \1_{\{\tau_r^{(u)} < t\}}).
    \end{multline*}
    By Lemmas~\ref{lem:supermartingale} and~\ref{lem:two_step_limit}, the above limit exists almost surely and is equal to $Z_\infty - Z_{\tau_\infty}$.
    By Lemma~\ref{lem:0_1_law}, almost surely, this limit is either $0$ or $Z_\infty$ depending on whether the integrability conditions~\eqref{eq:integrability_conditions} hold or not, which concludes.
\end{proof}

\section*{Acknowledgements}
The authors gratefully acknowledge financial support from the French National Research Agency through the ANR project ANR-24-CE40-1833 led by Bastien Mallein, which funded research visits to Toulouse and conference participation.
The authors warmly thank Michel Pain for helpful discussions throughout the project. J.B. also thanks Toulouse Jean Jaurès University and the IMT for a one-month invited professorship during which part of this work was undertaken.

\bibliographystyle{abbrv}
\bibliography{biblio}

\end{document}